\documentclass[reqno]{amsart}
\usepackage{mathptmx}
\usepackage{pxfonts}
\usepackage{amssymb}
\usepackage{amsfonts}
\usepackage{amsmath}
\usepackage{graphicx}
\usepackage{shadow}
\usepackage{color}
\usepackage[pagebackref]{hyperref}
\newtheorem{thm}{Theorem}[section]
\newtheorem{corollary}[thm]{Corollary}
\newtheorem{lemma}[thm]{Lemma}
\newtheorem{proposition}[thm]{Proposition}

\newtheorem{claim}{\sc Claim}

\theoremstyle{definition}
\newtheorem{definition}[thm]{Definition}
\newtheorem{example}[thm]{Example}

\theoremstyle{remark}
\newtheorem{remark}[thm]{Remark}

\newcommand{\field}[1]{\mathbb{#1}}

\newcommand{\Q }{\field{Q}}
\newcommand{\Z }{\field{Z}}
\newcommand{\N }{\field{N}}
\DeclareMathOperator{\Max}{Max}

\makeatletter
  \newcounter{xenumi}
  \newenvironment{xenumerate}{%
  \begin{list}{(\rm\arabic{xenumi})}{
    \setcounter{xenumi}{1}\usecounter{xenumi}
    \setlength{\parsep}{4\p@ \@plus2\p@ \@minus\p@}
    \setlength{\topsep}{6\p@ \@plus2\p@ \@minus2\p@}
    \setlength{\itemsep}{2\p@ \@plus1\p@ \@minus\p@}
    \setlength{\labelwidth}{0mm}
    \setlength{\labelsep}{2mm}
    \setlength{\itemindent}{2mm}
    \setlength{\leftmargin}{0mm}
    \setlength{\listparindent}{0mm}
  }}{\end{list}}
\makeatother

\begin{document}
%%%%%%%%%%%%%%%%%%%%%%%%%%%%%%%%%%%%%%%%%%%%%%%%%%%%%%%%%%%%%%%%%%%%%%%%%%%%%%%%%%%%%%%%%%%%%%%%%%%%%%%%%%%%%%%%%%%%%%%%%%%%%%%%%%%%%%%%%%%%%%
%%%%%%%%%%%%%%%%%%%%%%%%%%%%%%%%%%%%%%%%%%%%%%%%%%%%%%%%%%%%%%%%%%%%%%%%%%%%%%%%%%%%%%%%%%%%%%%%%%%%%%%%%%%%%%%%%%%%%%%%%%%%%%%%%%%%%%%%%%%%%%
%%%%%%%%%%%%%%%%%%%%%%%%%%%%%%%%%%%%%%%%%%%%%%%%%%%%%%%%%%%%%%%%%%%%%%%%%%%%%%%%%%%%%%%%%%%%%%%%%%%%%%%%%%%%%%%%%%%%%%%%%%%%%%%%%%%%%%%%%%%%%%
%%%%%%%%%%%%%%%%%%%%%%%%%%%%%%%%%%%%%%%%%%%%%%%%%%%%%%%%%%%%%%%%%%%%%%%%%%%%%%%%%%%%%%%%%%%%%%%%%%%%%%%%%%%%%%%%%%%%%%%%%%%%%%%%%%%%%%%%%%%%%%

\title[$t$-Reductions and $t$-integral closure of ideals]{$t$-Reductions and $t$-integral closure of ideals $^{(\star)}$}

\thanks{$^{(\star)}$ This work was supported by King Fahd University of Petroleum \& Minerals under DSR Grant \# RG1328.}

\author[S. Kabbaj]{S. Kabbaj}
\address{Department of Mathematics and Statistics, King Fahd University of Petroleum \& Minerals, Dhahran 31261, KSA}
\email{kabbaj@kfupm.edu.sa}

\author[A. Kadri]{A. Kadri}
\address{Department of Mathematics and Statistics, King Fahd University of Petroleum \& Minerals, Dhahran 31261, KSA}
\email{g201004080@kfupm.edu.sa}

\date{\today}


%\keywords{$t$-operation, $t$-ideal, $t$-invertibility, P$v$MD, Pr\"ufer domain, reduction of an ideal, integral closure of an ideal, $t$-reduction, $t$-integral dependence, basic ideal}

%\dedicatory{}

%%%%%%%%%%%%%%%%%%%%%%%%%%%%%%%%%%%%%%%%%%%%%%%%%%%%%%%%%%%%%%%%%%%%%%%%%%%%%%%%%%%%%%%%%%%%%%%%%%%%%%%%%%%%%%%%%%%%%%%%%%%%%%%%%%%%%%%%%%%%%%
%%%%%%%%%%%%%%%%%%%%%%%%%%%%%%%%%%%%%%%%%%%%%%%%%%%%%%%%%%%%%%%%%%%%%%%%%%%%%%%%%%%%%%%%%%%%%%%%%%%%%%%%%%%%%%%%%%%%%%%%%%%%%%%%%%%%%%%%%%%%%%
%%%%%%%%%%%%%%%%%%%%%%%%%%%%%%%%%%%%%%%%%%%%%%%%%%%%%%%%%%%%%%%%%%%%%%%%%%%%%%%%%%%%%%%%%%%%%%%%%%%%%%%%%%%%%%%%%%%%%%%%%%%%%%%%%%%%%%%%%%%%%%
%%%%%%%%%%%%%%%%%%%%%%%%%%%%%%%%%%%%%%%%%%%%%%%%%%%%%%%%%%%%%%%%%%%%%%%%%%%%%%%%%%%%%%%%%%%%%%%%%%%%%%%%%%%%%%%%%%%%%%%%%%%%%%%%%%%%%%%%%%%%%%
%%%%%%%%%%%%%%%%%%%%%%%%%%%%%%%%%%%%%%%%%%%%%%%%%%%%%%%%%%%%%%%%%%%%%%%%%%%%%%%%%%%%%%%%%%%%%%%%%%%%%%%%%%%%%%%%%%%%%%%%%%%%%%%%%%%%%%%%%%%%%%
%%%%%%%%%%%%%%%%%%%%%%%%%%%%%%%%%%%%%%%%%%%%%%%%%%%%%%%%%%%%%%%%%%%%%%%%%%%%%%%%%%%%%%%%%%%%%%%%%%%%%%%%%%%%%%%%%%%%%%%%%%%%%%%%%%%%%%%%%%%%%%
%%%%%%%%%%%%%%%%%%%%%%%%%%%%%%%%%%%%%%%%%%%%%%%%%%%%%%%%%%%%%%%%%%%%%%%%%%%%%%%%%%%%%%%%%%%%%%%%%%%%%%%%%%%%%%%%%%%%%%%%%%%%%%%%%%%%%%%%%%%%%%
\begin{abstract}
 Let $R$ be an integral domain and $I$ a nonzero ideal of $R$.  An ideal $J\subseteq I$ is a $t$-reduction of $I$ if $(JI^{n})_{t}=(I^{n+1})_{t}$ for some integer $n\geq0$. An element $x\in R$ is $t$-integral over $I$ if there is an equation $x^{n}+a_{1}x^{n-1}+...+a_{n-1}x+a_{n}=0$ with $a_{i}\in (I^{i})_{t}$ for $i=1,...,n$. The set of all elements that are $t$-integral over $I$ is called the $t$-integral closure of $I$. This paper investigates the $t$-reductions and $t$-integral closure of ideals. Our objective is to establish satisfactory $t$-analogues of well-known results, in the literature, on the integral closure of ideals and its correlation with reductions. Namely, Section 2 identifies basic properties of $t$-reductions of ideals and features explicit examples discriminating between the notions of reduction and $t$-reduction. Section 3 investigates the concept of $t$-integral closure of ideals, including its correlation with $t$-reductions. Section 4 studies the persistence and contraction of $t$-integral closure of ideals under ring homomorphisms. All along the paper, the main results are illustrated with original examples.
\end{abstract}
\maketitle

%%%%%%%%%%%%%%%%%%%%%%%%%%%%%%%%%%%%%%%%%%%%%%%%%%%%%%%%%%%%%%%%%%%%%%%%%%%%%%%%%%%%%%%%%%%%%%%%%%%%%%%%%%%%%%%%%%%%%%%%%%%%%%%%%%%%%%%%%%%%%%
%%%%%%%%%%%%%%%%%%%%%%%%%%%%%%%%%%%%%%%%%%%%%%%%%%%%%%%%%%%%%%%%%%%%%%%%%%%%%%%%%%%%%%%%%%%%%%%%%%%%%%%%%%%%%%%%%%%%%%%%%%%%%%%%%%%%%%%%%%%%%%
%%%%%%%%%%%%%%%%%%%%%%%%%%%%%%%%%%%%%%%%%%%%%%%%%%%%%%%%%%%%%%%%%%%%%%%%%%%%%%%%%%%%%%%%%%%%%%%%%%%%%%%%%%%%%%%%%%%%%%%%%%%%%%%%%%%%%%%%%%%%%%
%%%%%%%%%%%%%%%%%%%%%%%%%%%%%%%%%%%%%%%%%%%%%%%%%%%%%%%%%%%%%%%%%%%%%%%%%%%%%%%%%%%%%%%%%%%%%%%%%%%%%%%%%%%%%%%%%%%%%%%%%%%%%%%%%%%%%%%%%%%%%%
%%%%%%%%%%%%%%%%%%%%%%%%%%%%%%%%%%%%%%%%%%%%%%%%%%%%%%%%%%%%%%%%%%%%%%%%%%%%%%%%%%%%%%%%%%%%%%%%%%%%%%%%%%%%%%%%%%%%%%%%%%%%%%%%%%%%%%%%%%%%%%
%%%%%%%%%%%%%%%%%%%%%%%%%%%%%%%%%%%%%%%%%%%%%%%%%%%%%%%%%%%%%%%%%%%%%%%%%%%%%%%%%%%%%%%%%%%%%%%%%%%%%%%%%%%%%%%%%%%%%%%%%%%%%%%%%%%%%%%%%%%%%%
%%%%%%%%%%%%%%%%%%%%%%%%%%%%%%%%%%%%%%%%%%%%%%%%%%%%%%%%%%%%%%%%%%%%%%%%%%%%%%%%%%%%%%%%%%%%%%%%%%%%%%%%%%%%%%%%%%%%%%%%%%%%%%%%%%%%%%%%%%%%%%
\section{Introduction}

\noindent Throughout, all rings considered are commutative with identity. Let $R$ be a ring and $I$ an ideal of $R$. An ideal $J\subseteq I$ is a reduction of $I$ if $JI^{n}=I^{n+1}$ for some positive integer $n$. An ideal which has no reduction other than itself is called a basic ideal \cite{H1,H2,NR}. The notion of reduction was introduced by Northcott and Rees and its usefulness resides mainly in two facts: ``First, it defines a relationship between two ideals which is preserved under homomorphisms and ring extensions; secondly, what we may term the reduction process gets rid of superfluous elements of an ideal without disturbing the algebraic multiplicities associated with it" \cite{NR}. The main purpose of their paper was to contribute to the analytic theory of ideals in  Noetherian (local) rings via minimal reductions.

Reductions happened to be a very useful tool for the theory of integral dependence over ideals. Let $I$ be an ideal in a ring $R$. An element $x\in R$ is integral over $I$ if there is an equation $x^{n}+a_{1}x^{n-1}+...+a_{n-1}x+a_{n}=0$ with $a_{i}\in I^{i}$ for $i=1,...,n$. The set of all elements that are integral over $I$ is called the integral closure of $I$, and is denoted by $\overline{I}$. If $I=\overline{I}$, then $I$ is called integrally closed. It turned out that an element $x\in R$ is integral over $I$ if and only if $I$ is a reduction of $I+Rx$; and if $I$ is finitely generated, then $I\subseteq \overline{J}$ if and only if $J$ is a reduction of $I$ \cite[Corollary 1.2.5]{HS2}. This correlation allowed to prove a number of crucial results in the theory including the fact that the integral closure of an ideal is an ideal \cite[Corollary 1.3.1]{HS2}. For a full treatment of this topic, we refer the reader to Huneke and Swanson's book ``Integral closure of ideals, rings, and modules" \cite{HS2}.

Let $R$ be a domain with quotient field $K$, $I$ a nonzero fractional ideal of $R$, and let $I^{-1}:=(R:I)=\{x\in K\mid xI\subseteq R\}$. The $v$- and $t$-closures of $I$ are defined, respectively, by $I_v:=(I^{-1})^{-1}$ and $I_t:=\cup J_v$, where $J$ ranges over the set of finitely generated subideals of $I$. The ideal $I$ is a $v$-ideal (or divisorial) if $I_v=I$ and a $t$-ideal if $I_t=I$. Under the ideal $t$-multiplication $(I,J)\mapsto (IJ)_t$ the set $F_{t}(R)$ of fractional $t$~-ideals of $R$ is a semigroup with unit $R$. Recall that factorial domains, Krull domains, GCDs, and P$v$MDs can be regarded as $t$-analogues of the principal domains, Dedekind domains, B\'ezout domains, and Pr\"ufer domains, respectively. For instance, a domain is Pr\"ufer (resp., a P$v$MD) if every nonzero finitely generated ideal is invertible (resp., $t$-invertible). For some relevant works on $v$- and $t$-operations, we refer the reader to \cite{GHL,HZ,KM2007,KM2009,Kg2,P,Z1,Z2,Z4}.

 This paper investigates the $t$-reductions and $t$-integral closure of ideals. Our objective is to establish satisfactory $t$-analogues of well-known results, in the literature, on the integral closure of ideals and its correlation with reductions. Namely, Section 2 identifies basic properties of $t$-reductions of ideals and features explicit examples discriminating between the notions of reduction and $t$-reduction. Section 3 investigates the concept of $t$-integral closure of ideals, including its correlation with $t$-reductions. Section 4 studies the persistence and contraction of $t$-integral closure of ideals under ring homomorphisms. All along the paper, the main results are illustrated with original examples.

%%%%%%%%%%%%%%%%%%%%%%%%%%%%%%%%%%%%%%%%%%%%%%%%%%%%%%%%%%%%%%%%%%%%%%%%%%%%%%%%%%%%%%%%%%%%%%%%%%%%%%%%%%%%%%%%%%%%%%%%%%%%%%%%%%%%%%%%%%%%%%
%%%%%%%%%%%%%%%%%%%%%%%%%%%%%%%%%%%%%%%%%%%%%%%%%%%%%%%%%%%%%%%%%%%%%%%%%%%%%%%%%%%%%%%%%%%%%%%%%%%%%%%%%%%%%%%%%%%%%%%%%%%%%%%%%%%%%%%%%%%%%%
%%%%%%%%%%%%%%%%%%%%%%%%%%%%%%%%%%%%%%%%%%%%%%%%%%%%%%%%%%%%%%%%%%%%%%%%%%%%%%%%%%%%%%%%%%%%%%%%%%%%%%%%%%%%%%%%%%%%%%%%%%%%%%%%%%%%%%%%%%%%%%
%%%%%%%%%%%%%%%%%%%%%%%%%%%%%%%%%%%%%%%%%%%%%%%%%%%%%%%%%%%%%%%%%%%%%%%%%%%%%%%%%%%%%%%%%%%%%%%%%%%%%%%%%%%%%%%%%%%%%%%%%%%%%%%%%%%%%%%%%%%%%%
%%%%%%%%%%%%%%%%%%%%%%%%%%%%%%%%%%%%%%%%%%%%%%%%%%%%%%%%%%%%%%%%%%%%%%%%%%%%%%%%%%%%%%%%%%%%%%%%%%%%%%%%%%%%%%%%%%%%%%%%%%%%%%%%%%%%%%%%%%%%%%
%%%%%%%%%%%%%%%%%%%%%%%%%%%%%%%%%%%%%%%%%%%%%%%%%%%%%%%%%%%%%%%%%%%%%%%%%%%%%%%%%%%%%%%%%%%%%%%%%%%%%%%%%%%%%%%%%%%%%%%%%%%%%%%%%%%%%%%%%%%%%%
%%%%%%%%%%%%%%%%%%%%%%%%%%%%%%%%%%%%%%%%%%%%%%%%%%%%%%%%%%%%%%%%%%%%%%%%%%%%%%%%%%%%%%%%%%%%%%%%%%%%%%%%%%%%%%%%%%%%%%%%%%%%%%%%%%%%%%%%%%%%%%
\section{$t$-Reductions of ideals}\label{d}

\noindent This section identifies basic ideal-theoretic properties of the notion of $t$-reduction including its behavior under localizations. As a prelude to this, we provide explicit examples discriminating between the notions of reduction and $t$-reduction.

Recall that, in a ring $R$, a subideal $J$ of an ideal $I$ is called a reduction of $I$ if $JI^{n}=I^{n+1}$ for some positive integer $n$ \cite{NR}. An ideal which has no reduction other than itself is called a basic ideal \cite{H1,H2}.

%%%%%%%%%%%%%%%%%%%%%%%%%%%%%%%%%%%%%%%%%%%%%%%%%%%%%%%%%%%
%%%%%%%%%%%%%%%%%%%%%%%%%%%%%%%%%%%%%%%%%%%%%%%%%%%%%%%%%%%
\begin{definition}[{cf. \cite[Definition 1.1]{HKM}}]\label{d:def1}
Let $R$ be a domain and $I$ a nonzero ideal of $R$. An ideal $J\subseteq I$ is a \emph{$t$-reduction} of $I$ if $(JI^{n})_{t}=(I^{n+1})_{t}$ for some integer $n\geq0$ (and, a fortiori, the relation holds for $n\gg 0$). The ideal $J$ is a \emph{trivial $t$-reduction} of $I$ if $J_{t}=I_{t}$. The ideal $I$ is \emph{$t$-basic} if it has no $t$-reduction other than the trivial $t$-reductions.
\end{definition}

At this point, recall a basic property of the $t$-operation (which, in fact, holds for any star operation) that will be used throughout the paper. For any two nonzero ideals $I$ and $J$ of a domain, we have $(IJ)_{t}=(I_{t}J)_{t}=(IJ_{t})_{t}=(I_{t}J_{t})_{t}$. So, obviously, for nonzero ideals $J\subseteq I$, we always have:
$$J\ \text{is a $t$-reduction of}\ I \Leftrightarrow J\ \text{is a $t$-reduction of}\ I_{t} \Leftrightarrow  J_{t}\ \text{is a $t$-reduction of}\ I_{t}.$$
 Notice also that a reduction is necessarily a $t$-reduction; and the converse is not true, in general, as shown by the next example which exhibits a domain $R$ with two $t$-ideals $J\subsetneqq I$ such that $J$ is a $t$-reduction but not a reduction of $I$.

%%%%%%%%%%%%%%%%%%%%%%%%%%%%%%%%%%%%%%%%%%%%%%%%%%%%%%%%%%%
%%%%%%%%%%%%%%%%%%%%%%%%%%%%%%%%%%%%%%%%%%%%%%%%%%%%%%%%%%%
\begin{example}\label{d:exa1}
We use a construction from \cite{Hut}. Let $x$ be an indeterminate over $\Z$ and let $R := \Z[3x,x^{2},x^{3}]$, $I := (3x,x^{2},x^{3})$, and $J := (3x,3x^{2},x^{3},x^{4})$. Then $J\subsetneqq I$ are two finitely generated $t$-ideals of $R$ such that: $$JI^{n}\subsetneqq I^{n+1} \forall\ n\in\N\ \text{ and }\  (JI)_{t}=(I^{2})_{t}.$$
\end{example}

%%%%%%%%%%%%%%%%%%%%%%%%%%%
%%%%%%%%%%%%%%%%%%%%%%%%%%%
\begin{proof}
$I$, being a height-one prime ideal \cite{Hut}, is a $t$-ideal of $R$. Next, we prove that $J$ is a $t$-ideal. We first claim that
$J^{-1}=\frac{1}{x}\Z[x].$
Indeed, notice that $\Q(x)$ is the quotient field of $R$ and since $3x\subseteq J$, then $J^{-1}\subseteq \frac{1}{3x}R$. So, let $f:=\frac{g}{3x}\in J^{-1}$ where  $g=\sum^{m}_{i=0}a_{i}x^{i}\in \Z[x]$ with $a_{1}\in 3\Z$. Then the fact that $x^{3}f\in R$ implies that $a_{i}\in 3\Z$ for $i=0, 2, ..., m$; i.e., $g\in3\Z[x]$. Hence $f\in\frac{1}{x}\Z[x]$, whence $J^{-1}\subseteq\frac{1}{x}\Z[x]$. The reverse inclusion holds since
$\frac{1}{x}J\Z[x]=(3,3x,x^{2},x^{3})\Z[x]\subseteq R$,
proving the claim. Next, let $h\in (R:\Z[x])\subseteq R$. Then $xh\in R$ forcing $h(0)\in 3\Z$ and thus $h\in (3,3x,x^{2},x^{3})$. So, $(R:\Z[x])\subseteq (3,3x,x^{2},x^{3})$, hence $(R:\Z[x])=\frac{1}{x}J$. It follows that $J_{t}=J_{v}=\left(R:\frac{1}{x}\Z[x]\right)=x(R:\Z[x])=J$, as desired.

Next, let $n\in \N$. It is to see that $x^{3}x^{2n}=x^{2n+3}$ is the monic monomial with the smallest degree in $JI^{n}$. Therefore $x^{2(n+1)}=x^{2n+2}\in I^{n+1}\setminus JI^{n}$. That is, $J$ is not a reduction of $I$. It remains to prove $(JI)_{t}=(I^{2})_{t}$. We first claim that
$(JI)^{-1}=\frac{1}{x^{2}}\Z[x]$.
Indeed, $ (JI)^{-1}\subseteq (J^{-1})^{2}=\frac{1}{x^{2}}\Z[x]$
and the reverse inclusion holds since
$$\frac{1}{x^{2}}JI\Z[x]=(3,3x,x^{2},x^{3})(3,x,x^{2})\Z[x]\subseteq R$$
proving the claim. Now, observe that $I^{2}=(9x^{2},3x^{3},x^{4},x^{5})$. It follows that
$(IJ)_{t}=(IJ)_{v}=\left(R:\frac{1}{x^{2}}\Z[x]\right)=x^{2}(R:\Z[x])=xJ\supseteq  I^{2}$.
Thus $(IJ)_{t}\supseteq (I^{2})_{t}$, as desired.
\end{proof}

Observe that the domain $R$ in the above example is not integrally closed. Next, we provide a class of integrally closed domains where the notions of reduction and $t$-reduction are always distinct.

%%%%%%%%%%%%%%%%%%%%%%%%%%%
%%%%%%%%%%%%%%%%%%%%%%%%%%%
\begin{example}\label{d:exa3}
Let $R$ be any integrally closed Mori domain that is not completely integrally closed (i.e., not Krull). Then there always exist nonzero ideals $J\subsetneqq I$ in $R$ such that $J$ is a $t$-reduction but not a reduction of $I$.
\end{example}

%%%%%%%%%%%%%%%%%%%%%%%%%%%
%%%%%%%%%%%%%%%%%%%%%%%%%%%
\begin{proof}
These domains do exist; for instance, let $k\subsetneqq K$ be a field extension with $k$ algebraically closed and let $x$ be an indeterminate over $K$. Then, $R:=k + xK[x]$ is an  integrally closed Mori domain \cite[Theorem 4.18]{GH} that is not completely integrally closed \cite[Lemma 26.5]{G} (see \cite[p. 161]{FZ}).\\
Now, by \cite[Proposition 1.5(1)]{HKM}, there exists a $t$-ideal $A$ in $R$ that is not $t$-basic; say, $B\subseteq A$ is a $t$-reduction of $A$ with $B_t\subsetneqq A_t$. By \cite[Theorem 2.1]{Bar2}, there exist finitely generated ideals $F \subseteq A$ and $J \subseteq B$ such that $A^{-1}=F^{-1}$ and $B^{-1}=J^{-1}$; yielding $A_{t}=F_{t}$ and $B_{t}=J_{t}$. Let $I:= F+J$. Then, one can easily see, that $J$ is a non-trivial $t$-reduction of $I$. Finally, we claim that $J$ is not a reduction of $I$. Deny. Since $I$ is finitely generated, $I \subseteq \overline{J}$ by \cite[Corollary 1.2.5]{HS2}. But, $\overline{J}\subseteq J_{t}$ by \cite[Proposition 2.2]{Mi}. It follows that $J_{t} = I_{t}$, the desired contradiction.
\end{proof}

Another crucial fact concerns reductions of $t$-ideals. Indeed, if $J$ is a reduction of a $t$-ideal, then so is $J_{t}$; and the converse is not true, in general, as shown by the following example which features a domain $R$ with a $t$-ideal $I$ and an ideal $J\subseteq I$ such that $J_{t}$ is a reduction but $J$ is not a reduction of $I$.

%%%%%%%%%%%%%%%%%%%%%%%%%%%%%%%%%%%%%%%%%%%%%%%%%%%%%%%%%%%
%%%%%%%%%%%%%%%%%%%%%%%%%%%%%%%%%%%%%%%%%%%%%%%%%%%%%%%%%%%
\begin{example}\label{d:exa2}
Let $k$ be a field and let $x,y,z$ be indeterminates over $k$. Let $R := k[x]+M$, where $M := (y,z)k(x)[[y,z]]$ and let $J := M^{2}$. Note that $R$ is a classical pullback issued from the local Noetherian and integrally closed domain $T:=k(x)[[y,z]]$. Then $M$ is a divisorial ideal of $R$ by \cite[Corollary 5]{HKLM} and clearly, $\forall\ n\in\N$, $M^{n+2}\subsetneqq M^{n+1}$; that is, $J$ is not a reduction of $M$ in $R$. On the other hand, notice that $(M:M)=T$ (since $T$ is integrally closed) and $M$ is not principal in $T$. Therefore, by \cite[Theorem 13]{HKLM}, we have $$(R:(R:M^{2}))  =  (R:(M^{-1}:M)) =  (R:((M:M):M)) =$$ $$(R:(T:M)) =  (R:M^{-1}) =  M.$$
 So that $J_{t}=J_{v}=M$. Hence, $J_{t}$ is trivially a reduction of $M$ in $R$.
\end{example}

In the sequel, $R$ will denote a domain. For convenience, recall that, for any nonzero ideals $I,J,H$ of $R$, the equality $(IJ+H)_{t}=(I_{t}J+H)_{t}$ always holds since $I_{t}J\subseteq (I_{t}J)_{t}=(IJ)_{t}\subseteq (IJ+H)_{t}$. This property will be used in the proof of the next basic result which examines the $t$-reduction of the sum and product of ideals.

%%%%%%%%%%%%%%%%%%%%%%%%%%%%%%%%%%%%%%%%%%%%%%%%%%%%%%%%%%%
%%%%%%%%%%%%%%%%%%%%%%%%%%%%%%%%%%%%%%%%%%%%%%%%%%%%%%%%%%%
\begin{lemma}\label{b:lem2}
Let $J\subseteq I$ and $J'\subseteq I'$ be nonzero ideals of $R$. If $J$ and $J'$ are $t$-reductions of $I$ and $I'$, respectively, then $J+J'$ is a $t$-reduction of $I+I'$ and
$JJ'$ is a $t$-reduction of $II'$.
\end{lemma}

%%%%%%%%%%%%%%%%%%%%%%%%%%%
%%%%%%%%%%%%%%%%%%%%%%%%%%%
\begin{proof}
Let $n$ be a positive integer. Then the following implication always holds
\begin{equation}\label{b:lem2eq1} (JI^{n})_{t}=(I^{n+1})_{t}\ \Rightarrow\  (JI^{m})_{t}=(I^{m+1})_{t}\ \forall m\geq n.\end{equation}
Indeed, multiply the first equation through by $I^{m-n}$ and apply the $t$-closure to both sides. By (\ref{b:lem2eq1}), let $m$ be a positive integer such that
\begin{equation}\label{b:lem2eq2} (JI^{m})_{t}=(I^{m+1})_{t}\ \text{ and }\ (J'I'^{m})_{t}=(I'^{m+1})_{t}.\end{equation}
By (\ref{b:lem2eq2}), we get
$$\begin{array}{rcl}
\big((I+I')^{2m+1}\big)_{t}     &\subseteq  &\big(I^{m+1}(I+I')^{m} + I'^{m+1}(I+I')^{m}\big)_{t}\\
                                &=          &\big((I^{m+1})_{t}(I+I')^{m} + (I'^{m+1})_{t}(I+I')^{m}\big)_{t}\\
                                &=          &\big((JI^{m})_{t}(I+I')^{m} + (J'I'^{m})_{t}(I+I')^{m}\big)_{t}\\
                                &=          &\big(JI^{m}(I+I')^{m} + J'I'^{m}(I+I')^{m}\big)_{t}\\
                                &\subseteq  &\big((J+J')(I+I')^{2m}\big)_{t}\\
                                &\subseteq  &\big((I+I')^{2m+1}\big)_{t}
\end{array}$$
and then equality holds throughout, proving the first statement. The proof of the second statement is straightforward via (\ref{b:lem2eq2}).
\end{proof}

The next basic result examines the transitivity for $t$-reduction.

%%%%%%%%%%%%%%%%%%%%%%%%%%%%%%%%%%%%%%%%%%%%%%%%%%%%%%%%%%%
%%%%%%%%%%%%%%%%%%%%%%%%%%%%%%%%%%%%%%%%%%%%%%%%%%%%%%%%%%%
\begin{lemma}\label{b:lem3}
Let $K\subseteq J\subseteq I$ be nonzero ideals of $R$. Then:
\begin{enumerate}
\item If $K$ is a $t$-reduction of $J$ and $J$ is a $t$-reduction of $I$, then $K$ is a $t$-reduction of $I$.
\item If $K$ is a $t$-reduction of $I$, then $J$ is a $t$-reduction of $I$.
\end{enumerate}
\end{lemma}

%%%%%%%%%%%%%%%%%%%%%%%%%%%
%%%%%%%%%%%%%%%%%%%%%%%%%%%
\begin{proof}
For any positive integer $m$, we always have
\begin{equation}\label{b:lem3eq1} (JI^{m})_{t}=(I^{m+1})_{t}\ \Rightarrow\  (J^{n}I^{m})_{t}=(I^{m+n})_{t}\ \forall n\geq 1.\end{equation}
Indeed, multiply the first equation through by $J^{n-1}$, apply the $t$-closure to both sides, and conclude by induction on $n$. Let $(KJ^{n})_{t}=(J^{n+1})_{t}$ and $(JI^{m})_{t}=(I^{m+1})_{t}$, for some positive integers $n$ and $m$. By (\ref{b:lem3eq1}), we get
$$ (I^{m+n+1})_{t} = (J^{n+1}I^{m})_{t} = \big((J^{n+1})_{t}I^{m}\big)_{t} = \big((KJ^{n})_{t}I^{m}\big)_{t} = (KI^{m+n})_{t}$$
proving (a). The proof of (b) is straightforward.
\end{proof}

 The next basic result examines the $t$-reduction of the power of an ideal.

%%%%%%%%%%%%%%%%%%%%%%%%%%%%%%%%%%%%%%%%%%%%%%%%%%%%%%%%%%%
%%%%%%%%%%%%%%%%%%%%%%%%%%%%%%%%%%%%%%%%%%%%%%%%%%%%%%%%%%%
\begin{lemma} \label{b:lem4}
Let $J\subseteq I$ be nonzero ideals of $R$ and let $n$ be a positive integer. Then:
\begin{enumerate}
\item $J$ is a $t$-reduction of $I$ $\Leftrightarrow$ $J^{n}$  is a $t$-reduction of $I^{n}$.
\item  If $J=(a_{1}, ...,a_{k})$, then: $J$ is a $t$-reduction of $I$ $\Leftrightarrow$ $(a_{1}^{n}, ...,a_{k}^{n})$  is a $t$-reduction of $I^{n}$.
\end{enumerate}
\end{lemma}

%%%%%%%%%%%%%%%%%%%%%%%%%%%
%%%%%%%%%%%%%%%%%%%%%%%%%%%
\begin{proof}
(a) The ``only if" implication holds by Lemma~\ref{b:lem2}. For the converse, suppose $(J^{n}I^{nm})_{t}=(I^{nm+n})_{t}$ for some positive integer $m$. Then
$$(I^{nm+n})_{t}=(JJ^{n-1}I^{nm})_{t}\subseteq (JI^{nm+n-1})_{t} \subseteq (I^{nm+n})_{t}$$
and so equality holds throughout, as desired.

(b) Assume that $J$ is a $t$-reduction of $I$. From \cite[(8.1.6)]{HS2}, we always have the following equality
\begin{equation}\label{b:lem4eq1} \big(a_{1}^{n}, ...,a_{k}^{n}\big)\big(a_{1}, ...,a_{k}\big)^{(k-1)(n-1)}=\big(a_{1}, ...,a_{k}\big)^{(n-1)k+1} \end{equation}
and, multiplying (\ref{b:lem4eq1}) through by $J^{k-1}$, we get $\big(a_{1}^{n}, ...,a_{k}^{n}\big)J^{nk-n}=J^{nk}$. Therefore $(a_{1}^{n}, ...,a_{k}^{n})$ is a $t$-reduction of $J^{n}$ and a fortiori of $I^{n}$ by (a) and Proposition~\ref{b:lem3}. The converse holds by (a) and Proposition~\ref{b:lem3}.
\end{proof}

The next basic result examines the $t$-reduction of localizations.

%%%%%%%%%%%%%%%%%%%%%%%%%%%%%%%%%%%%%%%%%%%%%%%%%%%%%%%%%%%
%%%%%%%%%%%%%%%%%%%%%%%%%%%%%%%%%%%%%%%%%%%%%%%%%%%%%%%%%%%
\begin{lemma}\label{b:lem5}
Let $J\subseteq I$ be nonzero ideals of $R$ and let $S$ be a multiplicatively closed subset of $R$. If $J$ is a $t$-reduction of $I$, then
$S^{-1}J$ is a $t$-reduction of $S^{-1}I$.
\end{lemma}

%%%%%%%%%%%%%%%%%%%%%%%%%%%
%%%%%%%%%%%%%%%%%%%%%%%%%%%
\begin{proof}
Assume that $(JI^{n})_{t}=(I^{n+1})_{t}$ for some positive integer $n$. Let $t_{1}$ denote the $t$-operation with respect to $S^{-1}R$. By \cite[Lemma 3.4]{Kg2}, we have:\\
$((S^{-1}I)^{n+1})_{t_{1}}=(S^{-1}(I^{n+1}))_{t_{1}} = (S^{-1}((I^{n+1})_{t}))_{t_{1}} = (S^{-1}((JI^{n})_{t}))_{t_{1}} = (S^{-1}(JI^{n}))_{t_{1}}\\ = ((S^{-1}J)(S^{-1}I)^{n})_{t_{1}}$.
\end{proof}

It is worthwhile noting here that, in a P$v$MD, $J$ is a $t$-reduction of $I$ if and only if $J$ is $t$-locally a reduction of $I$; i.e.,  $JR_{M}$ is a reduction of $IR_{M}$ for every maximal $t$-ideal $M$ of $R$ \cite[Lemma 2.2]{HKM}.

%%%%%%%%%%%%%%%%%%%%%%%%%%%%%%%%%%%%%%%%%%%%%%%%%%%%%%%%%%%%%%%%%%%%%%%%%%%%%%%%%%%%%%%%%%%%%%%%%%%%%%%%%%%%%%%%%%%%%%%%%%%%%%%%%%%%%%%%%%%%%%
%%%%%%%%%%%%%%%%%%%%%%%%%%%%%%%%%%%%%%%%%%%%%%%%%%%%%%%%%%%%%%%%%%%%%%%%%%%%%%%%%%%%%%%%%%%%%%%%%%%%%%%%%%%%%%%%%%%%%%%%%%%%%%%%%%%%%%%%%%%%%%
%%%%%%%%%%%%%%%%%%%%%%%%%%%%%%%%%%%%%%%%%%%%%%%%%%%%%%%%%%%%%%%%%%%%%%%%%%%%%%%%%%%%%%%%%%%%%%%%%%%%%%%%%%%%%%%%%%%%%%%%%%%%%%%%%%%%%%%%%%%%%%
%%%%%%%%%%%%%%%%%%%%%%%%%%%%%%%%%%%%%%%%%%%%%%%%%%%%%%%%%%%%%%%%%%%%%%%%%%%%%%%%%%%%%%%%%%%%%%%%%%%%%%%%%%%%%%%%%%%%%%%%%%%%%%%%%%%%%%%%%%%%%%
%%%%%%%%%%%%%%%%%%%%%%%%%%%%%%%%%%%%%%%%%%%%%%%%%%%%%%%%%%%%%%%%%%%%%%%%%%%%%%%%%%%%%%%%%%%%%%%%%%%%%%%%%%%%%%%%%%%%%%%%%%%%%%%%%%%%%%%%%%%%%%
%%%%%%%%%%%%%%%%%%%%%%%%%%%%%%%%%%%%%%%%%%%%%%%%%%%%%%%%%%%%%%%%%%%%%%%%%%%%%%%%%%%%%%%%%%%%%%%%%%%%%%%%%%%%%%%%%%%%%%%%%%%%%%%%%%%%%%%%%%%%%%
\section{$t$-Integral closure of ideals}\label{p}

\noindent This section investigates the concept of $t$-integral closure of ideals and its correlation with $t$-reductions. Our objective is to establish satisfactory $t$-analogues of (and in some cases generalize) well-known results, in the literature, on the integral closure of ideals and its correlation with reductions.

%%%%%%%%%%%%%%%%%%%%%%%%%%%%%%%%%%%%%%%%%%%%%%%%%%%%%%%%%%%
%%%%%%%%%%%%%%%%%%%%%%%%%%%%%%%%%%%%%%%%%%%%%%%%%%%%%%%%%%%
\begin{definition}
Let $R$ be a domain and $I$ a nonzero ideal of $R$. An element $x\in R$ is $t$-integral over $I$ if there is an equation $$x^{n}+a_{1}x^{n-1}+...+a_{n-1}x+a_{n}=0\ \mbox{ with }\ a_{i}\in (I^{i})_{t}\ \forall i=1,...,n.$$ The set of all elements that are $t$-integral over $I$ is called the $t$-integral closure of $I$, and is denoted by $\widetilde{I}$. If $I=\widetilde{I}$, then $I$ is called $t$-integrally closed.
\end{definition}

 Notice that the $t$-integral closure of the ideal $R$ is always $R$, whereas the $t$-integral closure of the ring $R$ (also called pseudo-integral closure) may be larger than $R$; e.g., consider any non $v$-domain \cite{AHZ,FZ}. Also, we have  $J\subseteq I \Rightarrow \widetilde{J}\subseteq \widetilde{I}$. More ideal-theoretic properties are provided in Remark~\ref{p:rem1}.

It is well-known that the integral closure of an ideal is an ideal which is integrally closed \cite[Corollary 1.3.1]{HS2}. Next, we establish a $t$-analogue for this result.

%%%%%%%%%%%%%%%%%%%%%%%%%%%%%%%%%%%%%%%%%%%%%%%%%%%%%%%%%%%
%%%%%%%%%%%%%%%%%%%%%%%%%%%%%%%%%%%%%%%%%%%%%%%%%%%%%%%%%%%
\begin{thm}\label{p:main1}
The $t$-integral closure of an ideal is an integrally closed ideal. In general, it is not $t$-closed and, a fortiori, not $t$-integrally closed.
\end{thm}

The proof of this theorem relies on the following lemma which sets a $t$-analogue for the notion of Rees algebra of an ideal \cite[Chapter 5]{HS2}. Recall, for convenience, that the Rees algebra of an ideal $I$ (in a ring $R$) is the graded subring of $R[x]$ given by
$R[Ix] := \bigoplus_{n \geq 0} ~I^n~x^n$
 \cite[Definition 5.1.1]{HS2} and whose integral closure in $R[x]$ is the graded ring
$\bigoplus_{n \geq 0} ~\overline{I^n}~x^n$ \cite[Proposition 5.2.1]{HS2}.

%%%%%%%%%%%%%%%%%%%%%%%%%%%%%%%%%%%%%%%%%%%%%%%%%%%%%%%%%%%
%%%%%%%%%%%%%%%%%%%%%%%%%%%%%%%%%%%%%%%%%%%%%%%%%%%%%%%%%%%
\begin{lemma}\label{p:lem1}
Let $R$ be a domain, $I$ a $t$-ideal of $R$, and $x$ an indeterminate over $R$. Let $R_{t}[Ix] := \bigoplus_{n \geq 0} (I^n)_{t}x^n$. Then
$R_{t}[Ix]$ is a graded subring of $R[x]$ and its integral closure in $R[x]$ is the graded ring $\bigoplus_{n \geq 0} \widetilde{I^n}x^n.$
\end{lemma}

%%%%%%%%%%%%%%%%%%%%%%%%%%%
%%%%%%%%%%%%%%%%%%%%%%%%%%%
\begin{proof}
That $R_{t}[Ix]$ is $\N$-graded follows from the fact that $(I^i)_{t} \cdot(I^j)_{t} \subseteq (I^{i+j})_t, \forall i,j \in \N $. Let $\overline{R_{t}[Ix]}$ denote its integral closure in $R[x]$. By \cite[Theorem 2.3.2]{HS2}, $\overline{R_{t}[Ix]}$ is an $\N$-graded ring. Let  $k \in \N$ and let $S_k$ denote the homogeneous component of $\overline{R_{t}[Ix]}$ of degree $k$. We shall prove that $S_{k}= \widetilde{I^k} x^k$. Let $s:=s_{k}x^k\in S_k$, for some $s_k \in R$. Then, $s^n + a_{1}s^{n-1}+ \dots + a_n = 0$ for some positive integer $n$ and $a_i \in R_{t}[Ix]$, $i=1,\dots,n$. Expanding each $a_i = \sum_{j=0}^{k_{i}} a_{i,j}x^j$ with $a_{i,j} \in (I^j)_t$, the coefficient of the monomial of degree $kn$ in the above equation is $s_{k}^{n}+\sum_{i=1}^{n} a_{i,ik}s_{k}^{n-i}=0$, with $a_{i,ik} \in (I^{ik})_t$. It follows that $s_k \in \widetilde{I^{k}}$ and thus $S_{k} \subseteq \widetilde{I^k}x^k$. For the reverse inclusion, let $z_{k}:= y_{k} x^k\in \widetilde{I^k}x^k$, for some $y_{k}\in \widetilde{I^k}$. Then, $y_{k}^n + a_1 y_{k}^{n-1}+\dots+ a_n = 0$ for some positive integer $n$ and $a_j \in (I^{kj})_{t}$, $j=1,\dots,n$. Multiplying through by $x^{kn}$ yields the equation $z_{k}^n + a_1 x^{k} z_{k}^{n-1}+\dots+ a_n x^{kn} = 0$, with $a_{j}x^{kj} \in (I^{kj})_{t} x^{kj} \subseteq R_{t}[Ix]$, $j=1,\dots,n$. That is, $z_{k} \in \overline{R_{t}[Ix]}$. But $z_{k}$ is homogeneous of degree $k$ in $\overline{R_{t}[Ix]}$. Therefore $z_{k} \in S_{k}$ and hence $\widetilde{I^k} x^k\subseteq S_{k}$, completing the proof of the lemma.
\end{proof}

%%%%%%%%%%%%%%%%%%%%%%%%%%%%%%%%%%%%%%%%%%%%%%%%%%%%%%%%%%%
%%%%%%%%%%%%%%%%%%%%%%%%%%%%%%%%%%%%%%%%%%%%%%%%%%%%%%%%%%%
\begin{definition}
The $t$-Rees algebra of an ideal $I$ (in a domain $R$) is the graded subring of $R[x]$ given by $R_{t}[Ix] := \bigoplus_{n \geq 0} (I^n)_{t}x^n$.
\end{definition}

%%%%%%%%%%%%%%%%%%%%%%%%%%%%%%%%%%%%%%%%%%%%%%%%%%%%%%%%%%%%%%%%%%%%%%%%%%%%%%%%%%%%%%%%%%%%%%%%%%%%%%%%%%%%%%%
\begin{proof}[Proof of Theorem~\ref{p:main1}]
Let $R$ be a domain and $I$ a nonzero ideal of $R$. Since $\widetilde{I}=\widetilde{I_{t}}$, we may assume $I$ to be a $t$-ideal. We first prove that $\widetilde{I}$ is an ideal. Clearly, $\widetilde{I}$ is closed under multiplication. Next, we show that $\widetilde{I}$ is closed under addition. Let $a,b \in \widetilde{I}$. Then, by Lemma \ref{p:lem1}, $ax$ and $bx \in \overline{R_{t}[Ix]}$. Hence, $ax+bx = (a+b)x \in \overline{R_{t}[Ix]}$. Again, by Lemma \ref{p:lem1}, $a+b \in \widetilde{I}$, as desired. Next, we prove that $\widetilde{I}$ is integrally closed. For this purpose, observe that, $\forall n \in \N $, $(S_{1})^{n} \subseteq S_{n}$, forcing
\begin{equation}\label{power}\left(~\widetilde{I}~\right)^{n} \subseteq \widetilde{I^n}\ \forall n \in \N.\end{equation}
Consider the Rees algebra of the ideal $\widetilde{I}$, $R[\widetilde{I}x] = \bigoplus_{n \geq 0} \left(~\widetilde{I}~\right)^{n}x^n$. Therefore $R[\widetilde{I}x] \subseteq \overline{R_{t}[Ix]}$ and hence $\overline{R[\widetilde{I}x]} \subseteq \overline{R_{t}[Ix]}$. Now, a combination of Lemma~\ref{p:lem1} and \cite[Proposition 5.2.1]{HS2} yields $\bigoplus_{n \geq 0} \overline{\left(~\widetilde{I}~\right)^{n}}x^n \subseteq \bigoplus_{n \geq 0} \widetilde{I^{n}}x^n$. In particular, $\overline{~\widetilde{I}~} \subseteq \widetilde{I}$; that is, $\widetilde{I}$ is integrally closed. The proof of the last statement of the theorem is handled by Example~\ref{p:example2}(b), where we provide a domain with an ideal $I$ such that $\widetilde{I}\subsetneqq (~\widetilde{I}~)_{t}$. That is, $\widetilde{I}$ is not a $t$-ideal and, hence, not $t$-integrally closed since $(~\widetilde{I}~)_{t}\subseteq \widetilde{\widetilde{I}}$ always holds.
\end{proof}

The next result shows that the $t$-integral closure collapses to the $t$-closure in the class of integrally closed domains. It also completes two existing results in the literature on the integral closure of ideals (Gilmer \cite{G} and Mimouni \cite{Mi}).

%%%%%%%%%%%%%%%%%%%%%%%%%%%%%%%%%%%%%%%%%%%%%%%%%%%%%%%%%%%
%%%%%%%%%%%%%%%%%%%%%%%%%%%%%%%%%%%%%%%%%%%%%%%%%%%%%%%%%%%
\begin{thm}\label{p:main2}
Let $R$ be a domain. The following assertions are equivalent:
\begin{enumerate}
\item $R$ is integrally closed;
\item Every principal ideal of $R$ is integrally closed;
\item Every $t$-ideal of $R$ is integrally closed;
\item $\overline{I}\subseteq I_{t}$ for each nonzero ideal $I$ of $R$;
\item Every principal ideal of $R$ is $t$-integrally closed;
\item Every $t$-ideal of $R$ is $t$-integrally closed;
\item $\widetilde{I}=I_{t}$ for each nonzero ideal $I$ of $R$.
\end{enumerate}
\end{thm}

%%%%%%%%%%%%%%%%%%%%%%%%%%%
%%%%%%%%%%%%%%%%%%%%%%%%%%%
\begin{proof}
(a) $\Leftrightarrow$ (b) and (a) $\Leftrightarrow$ (c) $\Leftrightarrow$ (d) are handled by \cite[Lemma 24.6]{G} and \cite[Proposition 2.2]{Mi}, respectively. Also,  (g) $\Leftrightarrow$ (f) $\Rightarrow$ (e) $\Rightarrow$ (b) are straightforward. So, it remains to prove (a) $\Rightarrow$ (g). Assume $R$ is integrally closed and let $I$ be a nonzero ideal of $R$. The inclusion $I_{t} \subseteq \widetilde{I}$ holds in any domain. Next, let $\alpha \in \widetilde{I}$.

\begin{claim} There exists a finitely generated ideal $J \subseteq I$ such that $\alpha \in \widetilde{J}$. \end{claim}
Indeed, $\alpha$ satisfies an equation of the form $\alpha^{n}+a_{1}\alpha^{n-1}+...+a_{n}=0$ with $a_{i}\in (I^{i})_{t}\ \forall i=1,...,n$. Now, let $i\in\{1,..., n\}$. Hence, there exists a finitely generated ideal $F_{i}\subseteq I^{i}$ such that $a_{i}\in {F_{i}}_{v}$. Further, each generator of $F_{i}$ is a finite combination of elements of the form $\prod_{1\leq j\leq i} c_{j}\in I^{i}$. Let $J$ denote the subideal of $I$ generated by all $c_{j}$'s emanating from all $F_{i}$'s. Clearly, $a_{i}\in (J^{i})_{t}\ \forall i=1,...,n$. That is,  $\alpha \in \widetilde{J}$, proving the claim.

\begin{claim} $\widetilde{J} \subseteq J_{t}$. \end{claim}
Indeed, we first prove that $J^{-1} = (~\widetilde{J}~)^{-1}$. Clearly, $(~\widetilde{J}~)^{-1} \subseteq J^{-1}$. For the reverse inclusion, let $x \in J^{-1}$ and $y \in \widetilde{J}$. Then $y$ satisfies an equation of the form $y^{n}+a_{1}y^{n-1}+...+a_{n}=0$ with $a_{i}\in (J^{i})_{t}\ \forall i=1,...,n$. It follows that $(yx)^{n} + a_1 x (yx)^{n-1}+\dots+ a_n x^{n} = 0$ with $a_{i} x^{i} \in (J^{i})_{t}(J^{-1})^{i} \subseteq (J^{i})_{t} (J^{i})^{-1} = (J^{i})_{t} ((J^{i})_{t})^{-1} \subseteq R$. Hence $yx \in R$. Thus, $x \in (~\widetilde{J}~)^{-1}$, as desired. Therefore, $\widetilde{J} \subseteq (~\widetilde{J}~)_{v}= J_{v}=J_{t}$, proving the claim.

Now, by the above claims, we have $\alpha \in \widetilde{J}\subseteq J_{t}\subseteq I_{t}$. Consequently, $\widetilde{I}=I_{t}$, completing the proof of the theorem.
\end{proof}

In case all ideals of a domain are $t$-integrally closed, then it must be Pr\"ufer. This is a well-known result in the literature:

%%%%%%%%%%%%%%%%%%%%%%%%%%%%%%%%%%%%%%%%%%%%%%%%%%%%%%%%%%%
%%%%%%%%%%%%%%%%%%%%%%%%%%%%%%%%%%%%%%%%%%%%%%%%%%%%%%%%%%%
\begin{corollary}[{\cite[Theorem 24.7]{G}}]\label{p:cor-main2}
A domain $R$ is Pr\"ufer if and only if every ideal of $R$ is ($t$-)integrally closed.
\end{corollary}

 Now, we examine the correlation between the $t$-integral closure and $t$-reductions of ideals. In this vein, recall that, for the trivial operation, two crucial results assert that  \emph{$x \in \overline{I}$ $\Leftrightarrow$ $I$ is a reduction of $I+ Rx$} \cite[Corollary 1.2.2]{HS2}  and \emph{if $I$ is finitely generated and $J\subseteq I$, then:  $I\subseteq\overline{J}$   $\Leftrightarrow$ $J$ is a reduction of $I$} \cite[Corollary 1.2.5]{HS2}.  Next, we establish $t$-analogues of these two results.

%%%%%%%%%%%%%%%%%%%%%%%%%%%%%%%%%%%%%%%%%%%%%%%%%%%%%%%%%%%
%%%%%%%%%%%%%%%%%%%%%%%%%%%%%%%%%%%%%%%%%%%%%%%%%%%%%%%%%%%
\begin{proposition}\label{p:lem3}
Let $R$ be a domain and let $J\subseteq I$ be nonzero ideals of $R$.
\begin{enumerate}
\item $x \in \widetilde{I}$ $\Rightarrow$ $I$ is a $t$-reduction of $I+ Rx$.
\item Assume $I$ is finitely generated. Then: $I\subseteq\widetilde{J}$ $\Rightarrow$ $J$ is a $t$-reduction of $I$.
\end{enumerate}
Moreover, both implications are irreversible in general.
\end{proposition}

%%%%%%%%%%%%%%%%%%%%%%%%%%%
%%%%%%%%%%%%%%%%%%%%%%%%%%%
\begin{proof}
(a) Let $x \in \widetilde{I}$. Then, $x^n + a_1 x^{n-1}+\dots+ a_n = 0$ for some $a_i \in (I^i)_{t}$ for each $i \in \{1,\dots,n\}$. Hence
$$x^{n} \in I_{t} x^{n-1} + \dots + (I^{n})_{t} \subseteq \left(I_{t} x^{n-1} + \dots + (I^{n})_{t}\right)_{t} \subseteq \left(I(I+Rx)^{n-1}\right)_{t}.$$
It follows that $(I+Rx)^{n} \subseteq \left(I(I+Rx)^{n-1}\right)_{t}$. Hence, $((I+Rx)^{n})_{t} = \left(I(I+Rx)^{n-1}\right)_{t}$. Thus, $I$ is a $t$-reduction of $I+Rx$.

(b) Assume $I = (a_{1}, \ldots, a_{n})$, for some integer $n\geq 1$ and $a_{i} \in R\ \forall i =1, \ldots, n$. Suppose that $I \subseteq \widetilde{J}$. By (a), $J$ is a $t$-reduction of $J + Ra_{i}$, for each $i \in \{1,\ldots,n \}$. By Lemma~\ref{b:lem2}, $J$ is a $t$-reduction of $J+(a_{1}, \ldots, a_{n}) = I$, as desired.

The converse of (a) is not true, in general, as shown by Example~\ref{p:example2}(a). Also, (b) can be irreversible even with $I$ and $J$ both being finitely generated. For instance, consider the integrally domain $R$ of Example~\ref{d:exa3} with two ideals $J\subsetneqq I$, where $J$ is a non-trivial $t$-reduction of $I$ (i.e., $J_{t}\subsetneqq I_{t}$). By Theorem~\ref{p:main2}, $\widetilde{J}=J_{t}\nsupseteq I$.
\end{proof}

Next, we collect some ideal-theoretic properties of the integral closure of ideals.

%%%%%%%%%%%%%%%%%%%%%%%%%%%%%%%%%%%%%%%%%%%%%%%%%%%%%%%%%%%
%%%%%%%%%%%%%%%%%%%%%%%%%%%%%%%%%%%%%%%%%%%%%%%%%%%%%%%%%%%
\begin{remark}\label{p:rem1}
Let $R$ be a domain and let $I, J$ be nonzero ideals of $R$. Then:
\begin{xenumerate}
\item $I\subseteq \overline{I}\subseteq \widetilde{I}\subseteq\sqrt{I_{t}}$. Example~\ref{p:example1}(a) features a $t$-ideal for which these three containments are strict.  However, note that radical (and, a fortiori, prime) $t$-deals are necessarily $t$-integrally closed.

\item $\widetilde{~I\cap J~}\subseteq \widetilde{I}\cap \widetilde{J}$. The inclusion can be strict, for instance, in any integrally closed domain that is not a P$v$MD by \cite[Theorem 6]{A} and Theorem~\ref{p:main2}. Another example is provided in the non-integrally closed case by Example~\ref{p:example1}(c).

\item $\widetilde{I}+ \widetilde{J}\subseteq\widetilde{~I+J~}$. The inclusion can be strict. For instance, in $\Z[x]$, we have $\widetilde{(2)} + \widetilde{(x)} = (2,x)$ and $(2,x)^{-1}=\Z[x]$ so that $\widetilde{(2,x)}=(2,x)_{t}=\Z[x]$ (via Theorem~\ref{p:main2}).

\item By (\ref{power}), $\forall\ n\geq 1$, $(~\widetilde{I}~)^{n} \subseteq \widetilde{I^n}$. The inclusion can be strict, as shown by Example~\ref{p:example1}(b).

\item $\forall\ x\in R$, $x~\widetilde{I}\subseteq\widetilde{~xI~}$. Indeed, let $y \in x~\widetilde{I}$. Then, there is an equation of the form $y^n + (xa_1) y^{n-1}+\dots+ x^{n}a_n = 0$ with $x^{i}a_{i}  \in x^{i}(I^{i})_{t} = ((xI)^{i})_{t},\ i =1,\dots,n$. Hence, $y \in \widetilde{~xI~}$. Note that $x~\widetilde{I}=\widetilde{~xI~}$, $\forall\ x\in R$ and $\forall\ I$ ideal $\Leftrightarrow$ $R$ is integrally closed (Theorem~\ref{p:main2}).
\end{xenumerate}
\end{remark}

We close this section by the two announced examples.

%%%%%%%%%%%%%%%%%%%%%%%%%%%%%%%%%%%%%%%%%%%%%%%%%%%%%%%%%%%
%%%%%%%%%%%%%%%%%%%%%%%%%%%%%%%%%%%%%%%%%%%%%%%%%%%%%%%%%%%
\begin{example} \label{p:example1}
Let $R := \Z[\sqrt{-3}][2x, x^{2},x^{3}]$, $I := (2x^2, 2x^3, x^4, x^5)$, and $J := (x^3)$, where $x$ is an indeterminate over $\Z$. Then $I$ is a $t$-ideal of $R$ such that
\begin{enumerate}
\item $I \subsetneqq \overline{I} \subsetneqq \widetilde{I} \subsetneqq \sqrt{I}$.
\item $(~\widetilde{I}~)^2 \subsetneqq \widetilde{I^2}$.
\item $\widetilde{~J \cap I~} \subsetneqq \widetilde{J} \cap \widetilde{I}$.
\end{enumerate}
\end{example}

%%%%%%%%%%%%%%%%%%%%%%%%%%%
%%%%%%%%%%%%%%%%%%%%%%%%%%%
\begin{proof}
We first show that  $I$ is a $t$-ideal. Clearly, $\frac{1}{x^2} \Z[\sqrt{-3}][x] \subseteq I^{-1}$. For the reverse inclusion, let $f \in I^{-1} \subseteq x^{-4} R$. Then $f = \frac{1}{x^{4}} (a_{0}+a_{1}x + \dots + a_{n}x^{n})$ for some $n \in \N $, $a_{0} \in \Z [\sqrt{-3}]$, $a_{1}\in 2 \Z [\sqrt{-3}]$, and $a_{i} \in \Z [\sqrt{-3}]$ for $i \geq 2$. Since $2x^{2}f\in R$, then $a_{0} = a_{1} = 0$. It follows that $f \in \frac{1}{x^2} \Z[\sqrt{-3}][x]$. Therefore $I^{-1} = \frac{1}{x^{2}}\Z[\sqrt{-3}][x]$. Next, let $g \in (R:\Z[\sqrt{-3}][x]) \subseteq R$. Then $xg \in R$, forcing $g(0) \in 2\Z[\sqrt{-3}]$ and hence $g \in (2,2x,x^2,x^3)$. So $(R:\Z[\sqrt{-3}][x]) \subseteq (2,2x,x^{2},x^{3})$. The reverse inclusion is obvious. Thus, $(R:\Z[\sqrt{-3}][x]) = (2,2x,x^{2},x^{3})$. Consequently, we obtain $$I_{t} = I_{v} = (R: \frac{1}{x^{2}}\Z[\sqrt{-3}][x]) = x^{2} (R:\Z[\sqrt{-3}][x]) = I.$$

(a) Next, we prove the strict inclusions $I \subsetneqq \overline{I} \subsetneqq \widetilde{I} \subsetneqq \sqrt{I}$. For $I \subsetneqq \overline{I}$, notice that $(1+ \sqrt{-3})x^{2}\in \overline{I}\setminus I$ as $\big((1+ \sqrt{-3})x^{2}\big)^{3}=-8x^{6} \in I^{3}$ and $1+\sqrt{-3} \notin 2\Z[\sqrt{-3}]$.

 For $\overline{I} \subsetneqq \widetilde{I}$, we claim that $x^{3} \in \widetilde{I} \setminus \overline{I}$. Indeed, let $f\in (I^{2})^{-1}\subseteq x^{-8}R$. Then there are $n \in \N$, $a_{i} \in \Z[\sqrt{-3}]$ for $i \in \{0,2,\dots,n\}$, and $a_{1} \in 2\Z[\sqrt{-3}]$ such that $f=\frac{1}{x^{8}} (a_{0} + a_{1}x+ \dots + a_n x^{n})$. Since $4x^{4}f \in R$, then $a_{0} = a_{1} = a_{2} = a_{3}=0$. Therefore, $(I^{2})^{-1} \subseteq \frac{1}{x^{4}}\Z[\sqrt{-3}][x]$. The reverse inclusion is obvious. Hence, $(I^{2})^{-1} = \frac{1}{x^{4}}\Z[\sqrt{-3}][x]$. It follows that $$(I^2)_{t} = (I^2)_{v} = (R: \frac{1}{x^{4}}\Z[\sqrt{-3}][x]) = x^{4} (R:\Z[\sqrt{-3}][x]) = x^{2}I.$$
Hence $ x^{6} \in (I^2)_{t}$ and thus $x^{3} \in \widetilde{I}$. It remains to show that $x^{3} \not \in \overline{I}$. By \cite[Corollary 1.2.2]{HS2}, it suffices to show that $I$ is not a reduction of $I + (x^{3})$. Let $n \in \N $. It is easy to see that $x^{4} x^{3n}$ is the monic monomial with the smallest degree in $I\big(I+(x^{3})\big)^{n}$. Therefore, $x^{3(n+1)} = x^{3n+3} \in \big(I+(x^3)\big)^{n+1} \setminus I\big(I+(x^{3})\big)^{n}$. Hence, $I$ is not a reduction of $I + (x^{3})$, as desired.

For $\widetilde{I} \subsetneqq \sqrt{I}$, we  claim that $x^{2} \in \sqrt{I} \setminus \widetilde{I}$. Obviously, $x^{2} \in \sqrt{I}$. In order to prove that $x^{2} \not \in \widetilde{I}$, it suffices by Proposition~\ref{p:lem3} to show that $I$ is not a $t$-reduction of $I + (x^2)$. To this purpose, notice that $I+(x^2) = (x^2)$. Suppose by way of contradiction that  $(I(I+(x^2))^n)_{t} = ((I+(x^2))^{n+1})_{t}$ for some $n \in \N $. Then $(x^{2})^{n+1} = x^{2n+2} \in (I(I+(x^2))^n)_{t} = x^{2n}I$. Consequently, $x^{2} \in I$, absurd.

(b) We first prove that $\widetilde{I} = (2x^{2},(1+\sqrt{-3})x^{2}, x^{3}, x^{4})$. In view of (a) and its proof, we have $(2x^{2},(1+\sqrt{-3})x^{2}, x^{3}, x^{4})\subseteq\widetilde{I}$. Next, let $\alpha:=(a+b\sqrt{-3})x^{2}\in\widetilde{I}$ where $a,b \in \Z $. If $b = 0$,   then $a \neq 1$ as $x^{2} \not \in \widetilde{I}$.  Moreover, since $2x^2 \in \widetilde{I}$, $a$ must be even; that is, $\alpha\in (2x^2)$. Now assume $b\neq 0$. If $a = 0$, then $b \neq 1$ as $\sqrt{-3} x^{2}\not \in \widetilde{I}$. Moreover, since $2\sqrt{-3} x^{2} \in \widetilde{I}$, $b$ must be even; that is, $\alpha\in (2x^2)$. So suppose  $a \neq 0$. Then similar arguments force $a$ and $b$ to be of the same parity. Further, if  $a$ and $b$ are even, then $\alpha\in (2x^2)$; and if $a$ and $b$ are odd, then $\alpha\in  (2x^{2}, (1+\sqrt{-3})x^{2})$. Finally, we claim that $\widetilde{I}$ contains no monomials of degree 1. Deny and let $ax \in \widetilde{I}$, for some nonzero $a \in 2\Z[\sqrt{-3}]$. Then, by \cite[Remark 1.1.3(7)]{HS2}, $ax \in \widetilde{I} \subseteq \widetilde{(x^{2})} = \overline{(x^{2})}\subseteq\overline{x^{2}\Z[\sqrt{-3}][x]}$. By \cite[Corollary 1.2.2]{HS2}, $(x^{2})$ is a reduction of $(ax,x^2)$ in $\Z [\sqrt{-3}][x]$, absurd. Consequently,  $\widetilde{I} = (2x^{2},(1+\sqrt{-3})x^{2}, x^{3}, x^{4})$. Now, we are ready to check that $(~\widetilde{I}~)^2 \subsetneqq \widetilde{I^2}$. For this purpose, recall that $(I^2)_{t}= x^{2}I$. So, $2x^{4}\in \widetilde{I^2}$. We claim that $2x^{4}\notin (~\widetilde{I}~)^2$. Deny. Then, $2x^{4} \in (4x^{4}, 2(1+\sqrt{-3})x^{4})$, which yields $x^{2} \in (2x^{2}, (1+\sqrt{-3})x^{2}) \subseteq \widetilde{I}$, absurd.

(c) We claim that $x^3 \in \widetilde{I} \cap \widetilde{J} \setminus \widetilde{~I \cap J~}$. We proved in (a) that $x^{3} \in \widetilde{I}$. So, $x^{3} \in \widetilde{I} \cap \widetilde{J}$. Now, observe that $I \cap J = xI$ and assume, by way of contradiction, that $x^{3}\in \widetilde{~I \cap J~} = \widetilde{~xI~}$. Then $x^{3}$ satisfies an equation of the form $(x^{3})^{n} + a_{1} (x^{3})^{n-1} + \dots + a_{n} = 0$ with $a_{i} \in ((xI)^{i})_{t} = x^{i} (I^{i})_t,\ i=1,\dots,n$. For each $i$, let $a_{i}=x^{i} b_i$, for some $b_{i} \in (I^{i})_{t}$. Therefore  $(x^{2})^{n} + b_{1} (x^{2})^{n-1} + \dots + b_{n} = 0$. It follows that $x^{2} \in \widetilde{I}$, the desired contradiction.
\end{proof}

%%%%%%%%%%%%%%%%%%%%%%%%%%%%%%%%%%%%%%%%%%%%%%%%%%%%%%%%%%%
%%%%%%%%%%%%%%%%%%%%%%%%%%%%%%%%%%%%%%%%%%%%%%%%%%%%%%%%%%%
\begin{example}\label{p:example2}
Let $R:= \Z + x\Q (\sqrt{2})[x]$, $I:= (\frac{x}{\sqrt{2}})$, and $a:= \frac{x}{2}$, where $x$ is an indeterminate over $\Q$. Then:
\begin{enumerate}
\item $I$ is a $t$-reduction of $I+aR$ and  $a \not \in \widetilde{I}$.
\item $\widetilde{I} \subsetneqq (~\widetilde{I}~)_{t}$ and hence $\widetilde{I}\subsetneqq\widetilde{\widetilde{I}}$.
\end{enumerate}
\end{example}

%%%%%%%%%%%%%%%%%%%%%%%%%%%
%%%%%%%%%%%%%%%%%%%%%%%%%%%
\begin{proof}
(a)  First, we prove that $(I(I+aR))_{t} = ((I+aR)^{2})_{t}$. It suffices to show that $a^{2} \in (I(I+aR))_{t}$. For this purpose, let
$f\in(I(I+aR))^{-1} =(\frac{x^{2}}{2},\frac{x^{2}}{2\sqrt{2}})^{-1}\subseteq (\frac{x^{2}}{2})^{-1}= \frac{2}{x^{2}} R$. Then, $f = \frac{2}{x^{2}} (a_{0} + a_{1} x + \ldots + a_{n} x^{n})$, for some $n\geq 0$, $a_{0} \in \Z $, and $a_{i} \in \Q(\sqrt{2})$ for $i\geq 1$. Since $\frac{x^{2}}{2\sqrt{2}}f \in R$, $a_{0} = 0$.
It follows that $(I(I+aR))^{-1} \subseteq \frac{1}{x}  \Q (\sqrt{2})[x]$. On the other hand, $(I(I+aR)) (\frac{1}{x}\Q(\sqrt{2})[x])\subseteq R$.
So, we have
\begin{equation}\label{p:example2eq1}\big(I(I+aR)\big)^{-1} =\left(\frac{x^{2}}{2},\frac{x^{2}}{2\sqrt{2}}\right)^{-1}= \frac{1}{x}\Q(\sqrt{2})[x]\end{equation}
Now, clearly, $a^{2}(I(I+aR))^{-1}\subseteq R$. Therefore, $a^{2} \in (I(I+aR))_{v} = (I(I+aR))_{t}$, as desired.

Next, we prove that $a \not \in \widetilde{I}=\overline{I}$. By \cite[Corollary 1.2.2]{HS2}, it suffices to show that $I$ is not a reduction of $I+aR$. Deny and suppose that $I(I+aR)^{n} = (I+aR)^{n+1}$, for some positive integer $n$. Then $a^{n+1} = (\frac{x}{2})^{n+1} \in I(I+aR)^{n} = \frac{x}{\sqrt{2}}(\frac{x}{\sqrt{2}}, \frac{x}{2})^{n}$. One can check that this yields $1 \in \sqrt{2} (\sqrt{2},1)^{n}\subseteq (\sqrt{2})$ in $\Z[\sqrt{2}]$, the desired contradiction.

(b) We claim that $a \in (~\widetilde{I}~)_{t}$. Notice first that $x \in \widetilde{I}$ as $x^{2} \in I^{2}=(I^{2})_{t}$. Therefore, $A: = (x, \frac{x}{\sqrt{2}}) \subseteq \widetilde{I}$. Clearly, $A = \frac{2}{x}(\frac{x^{2}}{2}, \frac{x^{2}}{2 \sqrt{2}})$. Hence, by (\ref{p:example2eq1}), $A^{-1} = \Q (\sqrt{2})[x]$. However, $aA^{-1}\subseteq R$. Whence, $a \in A_{v} = A_{t} \subseteq (~\widetilde{I}~)_{t}$. Consequently, $a \in (~\widetilde{I}~)_{t} \setminus \widetilde{I}$.
\end{proof}

%%%%%%%%%%%%%%%%%%%%%%%%%%%%%%%%%%%%%%%%%%%%%%%%%%%%%%%%%%%%%%%%%%%%%%%%%%%%%%%%%%%%%%%%%%%%%%%%%%%%%%%%%%%%%%%%%%%%%%%%%%%%%%%%%%%%%%%%%%%%%%
%%%%%%%%%%%%%%%%%%%%%%%%%%%%%%%%%%%%%%%%%%%%%%%%%%%%%%%%%%%%%%%%%%%%%%%%%%%%%%%%%%%%%%%%%%%%%%%%%%%%%%%%%%%%%%%%%%%%%%%%%%%%%%%%%%%%%%%%%%%%%%
%%%%%%%%%%%%%%%%%%%%%%%%%%%%%%%%%%%%%%%%%%%%%%%%%%%%%%%%%%%%%%%%%%%%%%%%%%%%%%%%%%%%%%%%%%%%%%%%%%%%%%%%%%%%%%%%%%%%%%%%%%%%%%%%%%%%%%%%%%%%%%
%%%%%%%%%%%%%%%%%%%%%%%%%%%%%%%%%%%%%%%%%%%%%%%%%%%%%%%%%%%%%%%%%%%%%%%%%%%%%%%%%%%%%%%%%%%%%%%%%%%%%%%%%%%%%%%%%%%%%%%%%%%%%%%%%%%%%%%%%%%%%%
%%%%%%%%%%%%%%%%%%%%%%%%%%%%%%%%%%%%%%%%%%%%%%%%%%%%%%%%%%%%%%%%%%%%%%%%%%%%%%%%%%%%%%%%%%%%%%%%%%%%%%%%%%%%%%%%%%%%%%%%%%%%%%%%%%%%%%%%%%%%%%
%%%%%%%%%%%%%%%%%%%%%%%%%%%%%%%%%%%%%%%%%%%%%%%%%%%%%%%%%%%%%%%%%%%%%%%%%%%%%%%%%%%%%%%%%%%%%%%%%%%%%%%%%%%%%%%%%%%%%%%%%%%%%%%%%%%%%%%%%%%%%%
\section{Persistence and contraction of $t$-integral closure}\label{4}

Recall that the persistence and contraction of integral closure describe, respectively, the facts that for any ring homomorphism $\varphi: R\rightarrow T$, $\varphi(\overline{~I~})\subseteq \overline{\varphi(I)T}$ for every ideal $I$ of $R$, and $\overline{\varphi^{-1}(J)}=\varphi^{-1}(J)$ for every integrally closed ideal $J$ of $T$.

This section studies the persistence and contraction of $t$-integral closure. To this purpose, we first introduce the concept of $t$-compatible homomorphism which extends the well-known notion of $t$-compatible extension \cite{AEZ}. Throughout, we denote by $t$ (resp. $t_{1}$) and $v$ (resp. $v_{1}$) the $t$- and  $v$- closures in $R$ (resp., $T$).

%%%%%%%%%%%%%%%%%%%%%%%%%%%%%%%%%%%%%%%%%%%%%%%%%%%%%%%%%%%
%%%%%%%%%%%%%%%%%%%%%%%%%%%%%%%%%%%%%%%%%%%%%%%%%%%%%%%%%%%
\begin{lemma}\label{4:lem1}
Let $\varphi : R\longrightarrow T$  be a homomorphism of domains. Then, the following statements are equivalent:
\begin{enumerate}
\item $\varphi(I_{v})T \subseteq \big(\varphi(I)T\big)_{v_{1}}$, for each nonzero finitely generated ideal $I$ of $R$;
\item $\varphi(I_{t})T \subseteq \big(\varphi(I)T\big)_{t_{1}}$, for each nonzero ideal $I$ of $R$;
\item $\varphi^{-1}(J)$ is a $t$-ideal of $R$ for each $t_{1}$-ideal $J$ of $T$ such that $\varphi^{-1}(J) \neq 0$.
\end{enumerate}
\end{lemma}

%%%%%%%%%%%%%%%%%%%%%%%%%%%
%%%%%%%%%%%%%%%%%%%%%%%%%%%
\begin{proof}
(a) $\Rightarrow$ (c) Let $J$ be a $t_{1}$-ideal of $T$ and let $A$ be any finitely generated ideal of $R$ contained in $\varphi^{-1}(J)$. Then, $\varphi(A)T \subseteq J = J_{t_{1}}$. Further, $\varphi(A)T$ is finitely generated. Hence, $\big(\varphi(A)T\big)_{v_{1}} \subseteq J$. It follows, via (a), that $\varphi(A_{v})T \subseteq \big(\varphi(A)T\big)_{v_{1}} \subseteq J$. Therefore, $A_{v} \subseteq\varphi^{-1}(J)$ and thus $\varphi^{-1}(J)$ is a $t$-ideal.

(c) $\Rightarrow$ (b) Let $I$ be a nonzero ideal of $R$. The ideal $J := \big(\varphi(I)T\big)_{t_{1}}$ is clearly a $t_{1}$-ideal of $T$ with  $\varphi^{-1}(J) \neq 0$. By (c), $\varphi^{-1}(J)$ is a $t$-ideal of $R$. Consequently, we obtain
$$I_{t} \subseteq \Big(\varphi^{-1}(\varphi(I)T)\Big)_{t} \subseteq \Big(\varphi^{-1}\big(\varphi(I)T\big)_{t_{1}}\Big)_{t} = \big(\varphi^{-1}(J)\big)_{t} = \varphi^{-1}(J).$$
So that $\varphi(I_{t})T \subseteq J = \big(\varphi(I)T\big)_{t_{1}}$, as desired.

(b) $\Rightarrow$ (a) Trivial.
\end{proof}

%%%%%%%%%%%%%%%%%%%%%%%%%%%%%%%%%%%%%%%%%%%%%%%%%%%%%%%%%%%
%%%%%%%%%%%%%%%%%%%%%%%%%%%%%%%%%%%%%%%%%%%%%%%%%%%%%%%%%%%
\begin{definition}
A homomorphism of domains $\varphi : R\longrightarrow T$ is called $t$-compatible if it satisfies the equivalent conditions of Lemma~\ref{4:lem1}.
\end{definition}

When $\varphi$ denotes the natural embedding $R\subseteq T$, this definition matches the notion of $t$-compatible extension (i.e., $I_{t}T\subseteq (IT)_{t_{1}}$ for every ideal $I$ of $R$) well studied in the literature \cite{AEZ,BGR,E,FG}.

Next, we announce the main result of this section which establishes persistence and contraction of $t$-integral closure under $t$-compatible homomorphisms.

%%%%%%%%%%%%%%%%%%%%%%%%%%%%%%%%%%%%%%%%%%%%%%%%%%%%%%%%%%%
%%%%%%%%%%%%%%%%%%%%%%%%%%%%%%%%%%%%%%%%%%%%%%%%%%%%%%%%%%%
\begin{proposition} \label{4:main4}
Let $\varphi: R \longrightarrow T$ be a $t$-compatible homomorphism of domains, $I$ an ideal of $R$, and $J$ an ideal of $T$. Then:
\begin{enumerate}
\item $\varphi(~\widetilde{I}~)T \subseteq \widetilde{\varphi(I)T}$.
\item $\widetilde{\varphi^{-1}(J)} \subseteq \varphi^{-1}(~\widetilde{J}~)$. Moreover, if $J$ is $t$-integrally closed, then $\widetilde{\varphi^{-1}(J)}=\varphi^{-1}(J)$.
\end{enumerate}
\end{proposition}

%%%%%%%%%%%%%%%%%%%%%%%%%%%
%%%%%%%%%%%%%%%%%%%%%%%%%%%
\begin{proof}
(a) Let $x\in \widetilde{I}$, $y:= \varphi(x)$, and $z \in T$. We shall prove that $yz \in \widetilde{\varphi(I)T}$. Suppose that $x$ satisfies the equation $x^{n}+a_{1}x^{n-1}+...+a_{n}=0$ with $a_{i}\in (I^{i})_{t}$ for $i=1,...,n$. Then, apply $\varphi$ to this equation and multiply through by $z^{n}$ to obtain
$$(yz)^{n}+b_{1}z (yz)^{n-1}+...+b_{n-1}z^{n-1}(yz)+b_{n}z^{n}=0$$
 where $b_{i} := \varphi(a_{i}) \in \varphi((I^{i})_{t})T \subseteq (\varphi(I^{i})T)_{t_{1}} = \big(\big(\varphi(I)T\big)^{i}\big)_{t_{1}}$ by $t$-compatibility. Hence $b_{i} z^{i} \in \big(\big(\varphi(I)T\big)^{i}\big)_{t_{1}}$, for $i=1,...,n$. Consequently, $yz \in \widetilde{\varphi(I)T}$.

(b) Let $H:=\varphi(\varphi^{-1}(J))T$. Then, by (a), we have $$\varphi\big(\widetilde{\varphi^{-1}(J)}\big)T \subseteq \widetilde{H} \subseteq \widetilde{J}.$$
It follows that $\widetilde{\varphi^{-1}(J)}\subseteq \varphi^{-1}(\widetilde{J})$, as desired. Now, if $J$ is $t$-integrally closed, then $\widetilde{\varphi^{-1}(J)} \subseteq \varphi^{-1}(\widetilde{J}) = \varphi^{-1}(J) \subseteq \widetilde{\varphi^{-1}(J)}$ and hence the equality holds.
\end{proof}

 In the special case when both  $R$ and $T$ are integrally closed, persistence of $t$-integral closure coincides with $t$-compatibility by Theorem~\ref{p:main2}. This shows that the $t$-compatibility assumption in Proposition~\ref{4:main4} is imperative.

%%%%%%%%%%%%%%%%%%%%%%%%%%%%%%%%%%%%%%%%%%%%%%%%%%%%%%%%%%%
%%%%%%%%%%%%%%%%%%%%%%%%%%%%%%%%%%%%%%%%%%%%%%%%%%%%%%%%%%%
\begin{corollary}\label{4:extension}
Let $R \subseteq T$ be a $t$-compatible extension of domains and $I$ an ideal of $R$. Then:
\begin{enumerate}
\item $~\widetilde{I}T \subseteq \widetilde{IT}$.
\item $ \widetilde{I} \subseteq \widetilde{IT \cap R} \subseteq \widetilde{IT} \cap R$.
\end{enumerate}
Moreover, the above inclusions are strict in general.
\end{corollary}

%%%%%%%%%%%%%%%%%%%%%%%%%%%
%%%%%%%%%%%%%%%%%%%%%%%%%%%
\begin{proof}
(a) and (b) are direct consequences of Proposition~\ref{4:main4}. The inclusion in (a) and second inclusion in (b) can be strict as shown by Example~\ref{4:example3}.  The first inclusion in (b) can also be strict. For instance, let $R$ be an integrally closed domain and let $P \subsetneqq Q$ be prime ideals of $R$ with $x \in Q \setminus P$. Then $\widetilde{(x)}= (x)$ by Theorem~\ref{p:main2}. While $\widetilde{xR_{P} \cap R} = \widetilde{R_{P} \cap R} = R$. That is, $\widetilde{(x)} \subsetneqq \widetilde{(x)R_{P} \cap R}$.
\end{proof}

%%%%%%%%%%%%%%%%%%%%%%%%%%%%%%%%%%%%%%%%%%%%%%%%%%%%%%%%%%%
%%%%%%%%%%%%%%%%%%%%%%%%%%%%%%%%%%%%%%%%%%%%%%%%%%%%%%%%%%%
\begin{corollary} \label{4:loc}
Let $R$ be a domain, $I$ an ideal of $R$, and $S$ a multiplicatively closed subset of $R$. Then $S^{-1}\widetilde{I} \subseteq \widetilde{S^{-1}I}$.
\end{corollary}

%%%%%%%%%%%%%%%%%%%%%%%%%%%
%%%%%%%%%%%%%%%%%%%%%%%%%%%
\begin{proof}
It is well-known that flatness implies $t$-compatibility \cite[Proposition 0.6]{FG}. Hence, Corollary~\ref{4:extension} leads to the conclusion.
\end{proof}

For the integral closure, we always have $S^{-1}\overline{I}=\overline{S^{-1}I}$ \cite[Proposition 1.1.4]{HS2}. But in the above corollary the inclusion can be strict, as shown by the following example.

%%%%%%%%%%%%%%%%%%%%%%%%%%%%%%%%%%%%%%%%%%%%%%%%%%%%%%%%%%%
%%%%%%%%%%%%%%%%%%%%%%%%%%%%%%%%%%%%%%%%%%%%%%%%%%%%%%%%%%%
\begin{example}\label{4:example3}
We use a construction due to Zafrullah \cite{Z1}. Let $E$ be the ring of entire functions and $x$ an indeterminate over $E$. Let $S$ denote the set generated by the principal primes of $E$. Then, we claim that  $R := E+xS^{-1}E[x]$ contains a prime ideal $P$ such that $S^{-1} \widetilde{P} \subsetneqq \widetilde{S^{-1}P}$. Indeed, $R$ is a $P$-domain that is not a P$v$MD \cite[Example 2.6]{Z1}. By \cite[Proposition 3.3]{Z2}, there exists a prime $t$-ideal $P$ in $R$ such that $PR_{P}$ is not a $t$-ideal of $R_{P}$. By Theorem~\ref{p:main2}, we have $$\widetilde{P}R_{P} = PR_{P} \subsetneqq R_{p} = (PR_{P})_{t} = \widetilde{PR_{P}}$$ since $R$ is integrally closed. Also notice that $P=\widetilde{PR_{P}\cap R}\subsetneqq \widetilde{PR_{P}}\cap R=R$.
\end{example}

%%%%%%%%%%%%%%%%%%%%%%%%%%%%%%%%%%%%%%%%%%%%%%%%%%%%%%%%%%%
%%%%%%%%%%%%%%%%%%%%%%%%%%%%%%%%%%%%%%%%%%%%%%%%%%%%%%%%%%%
\begin{corollary}
Let $R$ be a domain and $I$ a $t$-ideal that is $t$-locally $t$-integrally closed (i.e., $I_{M}$ is $t$-integrally closed in $R_{M}$ for every maximal $t$-ideal $M$ of $R$). Then $I$ is $t$-integrally closed.
\end{corollary}

%%%%%%%%%%%%%%%%%%%%%%%%%%%
%%%%%%%%%%%%%%%%%%%%%%%%%%%
\begin{proof}
Let $\Max_{t}(R)$ denote the set of maximal $t$-ideals of $R$. By Corollary~\ref{4:loc}, we have
$$\begin{array}{lll}
 \widetilde{I}      &\subseteq      &\bigcap\limits_{M_{i} \in \Max_{t}(R)}(\widetilde{~I~})_{M_{i}}\\
                    &\subseteq      &\bigcap\limits_{M_{i} \in \Max_{t}(R)}\widetilde{I_{M_{i}}}\\
                    &=              &\bigcap\limits_{M_{i} \in \Max_{t}(R)} I_{M_{i}}\\
                    &=              &I.
\end{array}$$
Consequently, $I$ is $t$-integrally closed.
\end{proof}

\end{document}